\documentclass[letterpaper, 10 pt, conference]{ieeeconf}

\usepackage{cite}
\usepackage{amsmath,amssymb,amsfonts}
\usepackage{algorithmic}
\usepackage{graphicx}
\usepackage{textcomp}

\IEEEoverridecommandlockouts 
\usepackage[T1]{fontenc}
\usepackage[utf8]{inputenc}

\usepackage{url}
\usepackage{cite} 
\usepackage{amsmath,mathrsfs, textcomp, mathtools,amsthm}
\usepackage{color}
\usepackage{amssymb}
\usepackage[all]{xy}
\usepackage{float}
\usepackage{subcaption}
\usepackage[font=footnotesize,labelfont=footnotesize]{caption}
\usepackage{graphicx,xspace,bm}
\usepackage{mathtools}

\usepackage[lined,linesnumbered,algoruled]{algorithm2e}

\newtheorem{theorem}{Theorem}[section]
\newtheorem{proposition}[theorem]{Proposition}
\newtheorem{lemma}[theorem]{Lemma}

\theoremstyle{definition}

\newtheorem{remark}[theorem]{Remark}

\newtheorem{innercustomthm}{Algorithm}
\newenvironment{customthm}[1]
{\renewcommand\theinnercustomthm{#1}\innercustomthm}
{\endinnercustomthm}

\usepackage{xparse}

\newsavebox{\fminipagebox}
\NewDocumentEnvironment{fminipage}{m O{\fboxsep}}
{\par\kern#2\noindent\begin{lrbox}{\fminipagebox}
		\begin{minipage}{#1}\ignorespaces}
		{\end{minipage}\end{lrbox}%
	\makebox[#1]{%
		\kern\dimexpr-\fboxsep-\fboxrule\relax
		\fbox{\usebox{\fminipagebox}}%
		\kern\dimexpr-\fboxsep-\fboxrule\relax
	}\par\kern#2
}

\newcommand{\vertices}{\mathcal{V}}
\newcommand{\edges}{\mathcal{E}}

\newcommand{\norm}[1]{\ensuremath{\| #1 \|}}

\newcommand{\Bnorm}[1]{\ensuremath{\Big \| #1 \Big \|}}
\newcommand{\real}{{\mathbb{R}}}

\newcommand{\realnonnegative}{{\mathbb{R}}_{\ge 0}}

\newcommand{\integerspositive}{\mathbb{Z}_{\geq 1}}

\newcommand{\abs}[1]{\ensuremath{\left\lvert{#1}\right\rvert}}
\newcommand{\abss}[1]{\ensuremath{|{#1}|}}

\newcommand{\setdef}[2]{\{#1 \; | \; #2\}}
\newcommand{\setdefbig}[2]{\big\{#1 \; | \; #2\big\}}
\newcommand{\eps}{\epsilon}
\DeclareMathOperator*{\argmin}{\operatorname{argmin}}
\newcommand{\until}[1]{\{1,\dots,#1\}}
\newcommand{\map}[3]{#1:#2 \rightarrow #3}

\newcommand{\setr}[1]{\{#1\}}

\newcommand{\BB}{\mathcal{B}}

\newcommand{\II}{\mathcal{I}}

\renewcommand{\SS}{\mathcal{S}}

\newcommand{\YY}{\mathcal{Y}}
\newcommand{\ZZ}{\mathcal{Z}}
\newcommand{\FF}{\mathcal{F}}
\newcommand{\GG}{\mathcal{G}}

\newcommand{\HH}{\mathcal{H}}
\newcommand{\NN}{\mathcal{N}}

\newcommand{\EE}{\mathcal{E}}
\newcommand{\MM}{\mathcal{M}}
\newcommand{\PP}{\mathcal{P}}
\newcommand{\TT}{\mathcal{T}}
\newcommand{\VV}{\mathcal{V}}
\newcommand{\XX}{\mathcal{X}}

\newcommand{\ones}{\mathbf{1}}

\newcommand{\Eb}{\mathbb{E}}
\newcommand{\Pb}{\mathbb{P}}
\newcommand{\Qb}{\mathbb{Q}}
\newcommand{\Hb}{\mathbb{H}}

\newcommand{\Xihat}{\widehat{\Xi}}
\newcommand{\xihat}{\widehat{\xi}}
\newcommand{\uhat}{\widehat{u}}
\newcommand{\vhat}{\widehat{v}}
\newcommand{\data}[1]{\xihat^{\,#1}}

\newcommand{\PPhat}{\widehat{\PP}}
\newcommand{\aset}{\PPhat_L}
\newcommand{\Pbhat}{\widehat{\Pb}}

\newcommand{\st}{\operatorname{subject \text{$\, \,$} to}}
\renewcommand{\st}{\operatorname{s.} \operatorname{t.}}

\renewcommand{\until}[1]{[#1]}

\newcommand{\nuhat}{\widehat{\nu}}

\newcommand{\vag}{v^{\texttt{ag}}}
\newcommand{\xag}{x^{\texttt{ag}}}
\newcommand{\sag}{s^{\texttt{ag}}}
\newcommand{\CVaR}{\operatorname{CVaR}}
\newcommand{\Jdro}{J^{\star}_{\mathrm{DRO}}}
\newcommand{\Jref}{J^{\star}_{\mathrm{ref}}}
\renewcommand{\theenumi}{(\roman{enumi})}

\renewcommand{\baselinestretch}{0.985}

\newcommand{\oprocendsymbol}{\hbox{$\bullet$}}
\newcommand{\oprocend}{\relax\ifmmode\else\unskip\hfill\fi\oprocendsymbol}

\newcommand{\longthmtitle}[1]{\mbox{}\textup{\textsl{(#1):}}}

\usepackage[normalem]{ulem} 
\definecolor{new}{rgb}{0.55,0,0.55}
\usepackage[prependcaption,colorinlistoftodos]{todonotes}

\usepackage{xcolor}

\allowdisplaybreaks

\title{Data-driven distributionally robust optimization over a network via distributed semi-infinite programming}

\author{Ashish Cherukuri, Alireza Zolanvari, Goran Banjac, and Ashish R. Hota%
\thanks{A. Cherukuri and A. Zolanvari are with the Engineering and Technology Institute of Groningen, University of Groningen, The Netherlands (\texttt{\{a.k.cherukuri,a.zolanvari\}@rug.nl}), G. Banjac is with the Automatic Control Laboratory, ETH Z\"{u}rich, Switzerland (\texttt{gbanjac@ethz.ch}), and A. R. Hota is with the Department of Electrical Engineering, Indian Institute of Technology, Kharagpur, India (\texttt{ahota@ee.iitkgp.ac.in}).}}

\begin{document}

\maketitle

\begin{abstract}
	This paper focuses on solving a data-driven distributionally robust optimization problem over a network of agents. The agents aim to minimize the worst-case expected cost computed over a Wasserstein ambiguity set that is centered at the empirical distribution. The samples of the uncertainty are distributed across the agents. Our approach consists of reformulating the problem as a semi-infinite program and then designing a distributed algorithm that solves a generic semi-infinite problem that has the same information structure as the reformulated problem. In particular, the decision variables consist of both local ones that agents are free to optimize over and global ones where they need to agree on. Our distributed algorithm is an iterative procedure that combines the notions of distributed ADMM and the cutting-surface method. We show that the iterates converge asymptotically to a solution of the distributionally robust problem to any pre-specified accuracy.  Simulations illustrate our results.
\end{abstract}

\vspace*{-1ex}

\section{Introduction}\label{sec:intro}

Various machine learning problems can be cast as a stochastic optimization problem where the decision-maker intends to minimize an expected cost (termed commonly as the loss function). To solve this problem, the decision-maker usually has access to samples of the uncertainty modeled in the problem. In several applications, this sample set can be small, corrupted by noise, and distributed across several processors/agents. To handle the first two challenges, recent works have proposed solving a suitably defined data-driven distributionally robust optimization (DRO) problem~\cite{PME-DK:18}. The third challenge, of distributed data, is tackled by designing distributed algorithms, where agents iteratively compute decisions using only local communication, see e.g.,~\cite{GN-IN-AC:20,AN-JL:18}. In this paper, we wish to address all the three challenges together. We study data-driven DRO problems in a network setting where the data regarding the uncertainty is spread across multiple agents. We then design a distributed algorithm that solves this problem.

\subsubsection{Literature review}
 
A recent addition to the class of data-driven methods for solving stochastic optimization problems is data-driven DRO, where the decision-maker minimizes the worst-case expected cost computed over an ambiguity set of distributions that are close to the empirical distribution defined using data. Such an approach ensures that the resulting decision has better out-of-sample performance. Popular choices for metrics that define this closeness are KL-divergence~\cite{ZH-LJH:13}, $\phi$-divergence~\cite{RJ-YG:16}, Prohorov
metric~\cite{EE-GI:06}, and the Wasserstein metric~\cite{PME-DK:18,RG-AJK:16-arXiv}. In this paper, we focus on the latter; a survey of main ideas related to this method can be found in~\cite{DK-PME-VAN-SSA:19}. 

Much of the DRO literature focuses on solution algorithms that are centralized. In this work, we instead look into the distributed setting, where the data regarding the samples are not stored at a central location, but are instead distributed over multiple agents. The same setup was recently studied in~\cite{AC-JC:20-tac}, where a distributed primal-dual algorithm was proposed that solves the DRO problem. However, this work considered strong assumptions on the objective function (concavity with respect to the uncertainty over an identified domain) that need not hold for several statistical learning problems. In this work, we consider general objective functions that are locally Lipschitz and convex in the decision variable. 

\subsubsection{Setup and contributions}
Our approach builds on~\cite{FL-SM:19}, where the DRO problem is reformulated as a semi-infinite program. Noting the particular form of constraints and the objective of this problem, we form a general semi-infinite program that captures the particular case of DRO. In particular, the objective function of the problem has both global and separable local parts, and the constraints depend on both local and global decision variables. The way samples are distributed over the network, each agent is able to compute a subset of semi-infinite constraints. For such a problem structure, we design the distributed cutting-surface ADMM algorithm that solves the problem to any pre-specified accuracy. The algorithm combines features from the distributed ADMM method and the cutting-surface algorithm. We adapt our general algorithm to the DRO problem and highlight the convergence properties. 

We note that our algorithm solves a general semi-infinite problem in a distributed manner. The work~\cite{Burger:2012} also studies this problem and proposes a cutting-plane based distributed algorithm to solve it. However, therein, agents need to share cuts of the feasibility set and each cut is defined by a vector of the size of the primal variable. In our DRO problem, the size of the primal variable is of the order of the data points and is at least of the order of the number of agents. Thus, to implement their algorithm, agents need to communicate in each round a variable having size of the order of agents. This is not scalable.  We also note here that to determine the cuts in \cite{Burger:2012}, one needs to solve a maximization problem (defining the semi-infinite constraint) exactly. Since this problem is nonconvex in our setup, we encounter another hurdle. We overcome these roadblocks by focusing on cutting-surface algorithm~\cite{FL-SM:19} and developing a distributed version of it.

Our work is related to the broad field of federated learning, where machine learning problems are solved in a decentralized (with the presence of a central computing node) or distributed manner~\cite{JK-HBM-DR-PR:16-arXiv}. A recent related work~\cite{AR-FF-RP-AJ:20-arXiv} in this area looks into linear distributional shifts in data and designs a decentralized algorithm to solve this problem. In contrast, we consider more general distributional shifts and our algorithm does not require a central computing agent.

\noindent
\textbf{Notation:} Let $\real$, $\realnonnegative$, and $\integerspositive$ denote the set of real, nonnegative real, and positive integer numbers. For $n \in \integerspositive$, we denote $\until{n} := \{1, \dots , n\}$. The $2$-norm in $\real^n$ is represented by  $\norm{\cdot}$. We let $\ones_n=(1,\ldots,1) \in \real^n$ be the vector of ones. The $N$-fold Cartesian product $\SS \times \dots \times \SS$ of a set $\SS$ is denoted as $\SS^N$. 

\section{Problem statement}\label{sec:problem}

Consider $n \in \integerspositive$ agents communicating over an undirected connected graph $\GG := (\vertices,\edges)$.  The set of vertices are enumerated as $\VV:=\until{n}$ and $\EE \subset \VV \times \VV$ represents the set of edges.  Each agent $i \in \until{n}$ can send and receive information from its neighbors $\NN_i := \setdef{j \in \VV}{(i,j) \in \edges}$.  Let $\map{f}{\real^d \times \real^m}{\real} \colon (x,\xi) \mapsto f(x,\xi)$ be a locally Lipschitz objective function. Assume that for any $\xi \in \real^m$, the map $x \mapsto f(x,\xi)$ is convex. We are interested in solving the following \emph{data-driven distributionally robust optimization (DRO)} problem in a distributed manner over the network 
\begin{align}\label{eq:dro}
	\Jdro := \inf_{x \in \XX} \sup_{\Qb \in \aset} \Eb_{\Qb} [ f(x,\xi)],
\end{align}
where $\XX \subset \real^d$ is a compact convex set, $\xi \sim \Qb$ is the random variable, $\Eb_{\Qb}[ \, \cdot \,]$ denotes expectation with respect to the distribution $\Qb$, and $\aset$ is the set of distributions supported over a compact set $\Xi \subset \real^m$. Assume that $f$ is bounded on $\XX \times \Xi$. We refer to $\aset$ as the ambiguity set and we define it using data and the Wasserstein metric. In particular, assume we have $L \in \integerspositive$ samples of the uncertainty $\xi$, collected in the set $\Xihat := \{\xihat^1, \xihat^2, \dots, \xihat^L\}$. These samples need not be drawn from any particular distribution, however if they are sampled in an i.i.d manner from an underlying distribution, then we have desirable properties on the optimizer of the DRO problem~\eqref{eq:dro}, as illustrated in~\cite{PME-DK:18}. Let $\Pbhat_L := \frac{1}{L} \sum_{\ell=1}^L \delta_{\xihat^\ell}$ be the empirical distribution, that is, the distribution defined as the summation of delta functions of equal mass placed on each of the samples. Then, the ambiguity set $\aset$ is given as
\begin{align*}
  \aset := \setdef{\Qb \in \MM}{d_W(\Qb,\Pbhat_L) \le \theta},
\end{align*}
where $\MM$ is the set of distributions supported on $\Xi \subset \real^m$, $d_W$ is the $1$-Wasserstein metric on the space of distributions, and $\theta > 0$ is the radius. Note that given two distributions $\Pb_1$ and $\Pb_2$ supported on $\Xi$, the $1$-Wasserstein distance between them (defined using the Euclidean norm) is given as
\begin{align*}
	d_W(\Pb_1,\Pb_2) := \min_{\Hb \in \HH(\Pb_1,\Pb_2)} \left\{ \int_{\Xi \times \Xi} \norm{\xi-\omega} \Hb(d \xi, d \omega) \right\},
\end{align*}
where $\HH(\Pb_1,\Pb_2)$ is the set of all distributions on $\Xi \times \Xi$ with marginals $\Pb_1$ and $\Pb_2$.

The objective of the paper is to design an algorithm for a group of agents to solve the DRO problem~\eqref{eq:dro} in a distributed manner. To this end, we assume that each agent only knows a certain number (at least one) of samples from the set $\Xihat$. Denoting the data available to agent $i$ by the set $\Xihat_i \subset \Xihat$, we assume that $\Xihat_i \cap \Xihat_j = \emptyset$ for all $i,j \in \until{n}$ and $\Xihat= \cup_{i=1}^n \Xihat_i$. Each agent knows the function $f$, the sets $\XX$ and $\Xi$, and the radius of the ambiguity set $\theta > 0$. The challenge comes from the fact that each agent only has limited access to the data.

\subsubsection{Motivating example}
We now briefly motivate our problem setup. Consider a statistical learning problem where we aim to minimize the expected value of a loss function 
\begin{align}\label{eq:loss}
	\min_{x \in \XX} \, \, \Eb_{\Pb} [f(x,\xi)],
\end{align}
where $x$ is the decision variable and $\xi$ is the uncertainty that stands for the input-output data. Roughly speaking, the value $f(x,\xi)$ encodes the ability of $x$ to predict the relationship between the input-output pair $\xi$. A lower value of $f(x,\xi)$ means higher prediction accuracy. Thus, one seeks to minimize the expected value of this uncertain function. Usually, the distribution $\Pb$ is unknown and instead a data set $\Xihat = \{\xihat^1,\dots,\xihat^L\}$ possibly sampled in an i.i.d manner from $\Pb$ is available. Then, the expected value in~\eqref{eq:loss} is replaced with the sample average and one seeks the solution of
\begin{align}\label{eq:average-loss}
	\min_{x \in \XX} \, \, \textstyle \frac{1}{L} \sum_{\ell=1}^L f(x,\xihat^\ell).
\end{align} 
Now examine the network setup as explained above where samples $\Xihat$ are distributed across agents in the network. Then, the minimization problem~\eqref{eq:average-loss} is equivalently written as the distributed optimization problem
\begin{align}\label{eq:distributed}
	\min_{x \in \XX} \, \, \textstyle \sum_{i = 1}^n \varphi_i(x),
\end{align} 
where $\varphi_i$ is known to agent $i$ and it takes the form $\varphi_i(x) = \frac{1}{L} \sum_{\setdef{\ell \in [L]}{\xihat^\ell \in \Xihat_i}} f(x,\xihat^\ell)$. This formulation fits naturally into the traditional distributed optimization setup where agents across the network seek to minimize the sum of cost functions~\cite{GN-IN-AC:20,AN-JL:18}. While the solution of the sample average problem~\eqref{eq:distributed} has desirable statistical guarantees, such as out-of-sample performance, when the dataset is large and is sampled from $\Pb$, it often fails to demonstrate such properties when the dataset is small or encounters distributional shifts. The solution of the DRO problem~\eqref{eq:dro} fixes these issues. Thus, one would like to solve~\eqref{eq:dro} in a distributed manner. Unfortunately, the problem~\eqref{eq:dro} cannot be written equivalently in a similar form as~\eqref{eq:distributed}, and so we need to rethink the design of a distributed algorithm. This motivates our current work.

\subsubsection{Semi-infinite reformulation and cutting-surface method}\label{sec:sip-reform}
Note that~\eqref{eq:dro} consists of an infinite-dimensional inner optimization problem. Following~\cite{FL-SM:19}, we reformulate this as a semi-infinite program. Then, problem~\eqref{eq:dro} can be solved equivalently by finding a solution of
\begin{equation}\label{eq:dro-sip}
  \Jref := \left\{ \begin{array}{cl}
	  \underset{x,s,v}{\min} & \displaystyle \ones_L^\top v + L\theta s 
	  \\
	  \st & v \in \real^{L}, \quad s \in \realnonnegative, \quad x \in \XX, 
    \\
    & f(x,\xi) - v_\ell - s \norm{\xi - \data{\ell}} \le 0, 
    \\
    & \quad \quad  \quad \quad  \quad \quad \quad \forall \xi \in \Xi,  \;\; \ell \in \until{L},
  \end{array} \right.
\end{equation}
where $v_\ell$ represents the $\ell$-th component of the vector $v$ and we recall that $\ones_L$ is the vector of ones. The equivalence here means that $x$ is a solution of~\eqref{eq:dro} if and only if there exists $(v,s)$ such that $(x,v,s)$ is a solution of \eqref{eq:dro-sip}. However, it is important to note that the optimal value of~\eqref{eq:dro-sip} is $L$ times that of~\eqref{eq:dro}. That is, $\Jref = L \Jdro$. The decision variable of the above problem has size proportional to the number of samples. Further, we have $L$ semi-infinite constraints.

The above optimization problem is a particular case of a general convex semi-infinite program that can be written as 
	\begin{align}\label{eq:gen-sip-lit}
		\min_{x \in \XX} \, \, \setdef{\varphi(x)}{g(x,\xi) \le 0, \text{ for all } \xi \in \Xi},
	\end{align}
	where $\XX \subset \real^d$ and $\Xi \subset \real^m$ are compact, $\XX$ is convex, $\map{\varphi}{\real^d}{\real}$ is convex, $g$ is continuous and bounded on $\XX \times \Xi$, and $\map{g(\cdot,\xi)}{\real^d}{\real}$ is convex for every $\xi \in \Xi$. Semi-infinite programs are in general computationally challenging to solve, see the survey~\cite{RH-KOK-93}. We specifically focus our attention on the cutting-surface algorithm given in~\cite{FL-SM:19} as it naturally allows a distributed implementation. The cutting-surface algorithm consists of the following steps:
\begin{enumerate}
	\item Start with $x^0 \in \XX$ and a set $\Xi^0 = \emptyset$ that maintains cuts. A cut is a point in the uncertainty set $\Xi$. 
	\item At iteration $k$, given a finite set of cuts $\Xi^k \subset \Xi$, solve 
	\begin{align*}
		\min_{x \in \XX} \, \, \setdef{\varphi(x)}{g(x,\xi) \le 0, \text{ for all } \xi \in \Xi^k},
	\end{align*}
	and store the solution as $x^{k+1}$. 
	\item Find a new cut $\xi^{k+1}$ corresponding to $x^{k+1}$ that satisfies two properties. First, it is an $\eps/2$-optimal solution for the problem $\max_{\xi \in \Xi} g(x^{k+1},\xi)$. That is, 
	\begin{align*}
		g(x^{k+1},\xi^{k+1}) \ge \max_{\xi \in \Xi} g(x^{k+1},\xi) - \frac{\eps}{2}.
	\end{align*}
	Second, $g(x^{k+1},\xi^{k+1}) > \eps/2$. Note that a cut $\xi^{k+1}$ satisfying the above two properties exists as long as $x^{k+1}$ violates the robust constraint by at least $\eps$. That is, $\max_{\xi \in \Xi} g(x^{k+1},\xi) > \eps$. Once the cut $\xi^{k+1}$ is found, update the cut set as $\Xi^{k+1} = \Xi^{k} \cup \{\xi^{k+1}\}$ and move to the next iteration $k+1$.
	\item If no such cut is found, then declare $x^{k+1}$ as the desired solution and terminate the procedure. 
\end{enumerate}
As established in~\cite{FL-SM:19}, this algorithm converges in a finite number of steps to an $\eps$-feasible solution $x^\star \in \XX$ satisfying $\varphi(x^\star) \le \varphi^\star$, where $\varphi^\star$ is the optimal value of~\eqref{eq:gen-sip-lit}. Here $\eps$-feasibility means that the violation of the robust constraint is no more than $\eps$, i.e., $\max_{\xi \in \Xi} g(x^\star,\xi) \le \eps$.  Our aim is to develop a distributed version of the above algorithm. 

In the following section, we design a distributed algorithm that solves a generic semi-infinite program, for which~\eqref{eq:dro-sip} is a special case. Subsequently, in Section~\ref{sec:dist-dro}, we adapt this distributed algorithm to solve~\eqref{eq:dro-sip}.

\section{Distributed cutting-surface ADMM for convex semi-infinite program}\label{subsec:generic}

In this section, we provide a distributed algorithm to solve the following semi-infinite optimization problem 
\begin{equation}\label{eq:gen-sip}
  \begin{array}{cl}
	  \underset{y,\setr{z_i}_{i=1}^n}{\min} & \displaystyle \sum_{i=1}^n \left(\varphi_i(z_i) + h_i(y)\right) 
	  \\
	  \st & y \in \YY, 
	  \\
	  & z_i \in \ZZ_i, \quad \forall i \in \until{n},
	  \\
    & g_i(y,z_i,\xi) \le 0, \quad \forall \xi \in \Xi, \quad i \in \until{n}.
  \end{array}
\end{equation}
Here, functions $\{\varphi_i,h_i\}_{i=1}^n$ are convex, sets $\YY \subset \real^d$, $\ZZ_i \subset \real^{p_i}$, $i \in \until{n}$ are convex compact, and $\Xi \subset \real^m$ is compact. Further, $\varphi_i$ and $h_i$ are bounded on $\ZZ_i$ and $\YY$, respectively, and $g_i$ is locally Lipschitz and convex in $(y,z_i)$ for any fixed $\xi \in \Xi$. The problem~\eqref{eq:gen-sip} is convex under these conditions. Before we move on to the distributed algorithm, we comment about the information structure and the connection to the reformulated DRO problem~\eqref{eq:dro-sip}. We assume that each agent $i$ only knows $\varphi_i$, $h_i$, $g_i$, $\YY$, $\ZZ_i$, and $\Xi$.
The variable $z_i \in \real^{p_i}$ is a local decision variable for agent $i$ while $y \in \real^d$ is a global variable that all agents need to agree on. Note that the semi-infinite constraint for each agent $i$, $g_i(y,z_i,\xi) \le 0$ for all $\xi \in \Xi$, depends on both the local and the global variable. 
Comparing the two, one observes that problem~\eqref{eq:dro-sip} is a special case of~\eqref{eq:gen-sip} with decision variables $z_i=(v_\ell)_{\setdef{\ell}{\data{\ell} \in \Xihat_i}} \in \real^{\abss{\Xihat_i}}$, $y=(x,s) \in \real^{d+1}$, and objective functions $\varphi_i(z_i)=\ones_{\abss{\Xihat_i}}^\top z_i$, $h_i(y)= \abss{\Xihat_i} \theta s$. The parallelism between constraints is more involved and we will address this point in the subsequent section. 

Our distributed algorithm builds on a class of distributed ADMM algorithms where auxiliary variables are used for each edge of the network to impose consensus, see e.g.,~\cite{Banjac:2019,GM-JAB-GBG:10,AM-AO:17}. Next, we overview the derivation of such an algorithm. Let $\setr{y_i}_{i \in \until{n}}$ be local estimates of $y$ maintained by agents and let $\{u_{ij}\}_{(i,j) \in \EE}$ be the auxiliary variables that will ensure consensus over the global decision variable $y$. Using these variables, one can write~\eqref{eq:gen-sip} equivalently as 
\begin{equation}\label{eq:gen-sip-reform}
  \begin{array}{cl}
	  \underset{\substack{\setr{(y_i,z_i)}_{i=1}^n, \\ \setr{u_{ij}}_{(i,j) \in \EE}}}{\min} & \displaystyle \sum_{i=1}^n \left( \varphi_i(z_i) +  h_i(y_i) \right) 
	  \\
	  \st & y_i \in \YY, \quad \forall i \in \until{n}, 
	  \\
	& z_i \in \ZZ_i, \quad \forall i \in \until{n},
	\\
    & g_i(y_i,z_i,\xi) \le 0, \quad \forall \xi \in \Xi, \quad i \in \until{n},
    \\
    &  y_i = u_{ij}, \quad \forall i \in \until{n}, \, (i,j) \in \EE,
    \\
    & y_j = u_{ij}, \quad \forall j \in \until{n}, \, (i,j) \in \EE .
  \end{array}
\end{equation}
Note that constraints $y_i = u_{ij} = y_j$ in the above formulation impose consensus on $y$. Next, we derive a distributed ADMM algorithm using the reformulation~\eqref{eq:gen-sip-reform}. The derivation is given for completeness and can be found both in~\cite{Banjac:2019} and~\cite{GM-JAB-GBG:10}.  In contrast with the optimization problems considered in these works, the problem~\eqref{eq:gen-sip-reform} has semi-infinite constraints for each agent and we only require consensus over a subset of primal variables.  For now, we do not deal with the computational issues related to semi-infinite constraints. We will revisit this point once we have derived the distributed ADMM. 
The augmented Lagrangian corresponding to the problem~\eqref{eq:gen-sip-reform} is
\begin{align*}
	& L_\rho(Y, Z, U, \Lambda,\Gamma) 
	\\
	& \qquad := \sum_{i=1}^n \Bigl[ \varphi_i(z_i) + h_i(y_i)  +  \textstyle \sum_{j \in \NN_i} \bigl( \lambda_{ij}^\top (y_i - u_{ij}) 
		\\
& \qquad \tfrac{\rho}{2} \norm{y_i - u_{ij}}^2 + \gamma_{ij}^\top (y_j - u_{ij}) + \tfrac{\rho}{2} \norm{y_j - u_{ij}}^2 \bigr) \Bigr],
\end{align*}
where $\rho > 0$ is a parameter, $Y := (y_i)_{i\in\until{n}}$, $Z := (z_i)_{i\in\until{n}}$, and $U:=(u_{ij})_{(i,j) \in \EE}$ are the primal variables and $\Lambda := (\lambda_{ij})_{(i,j) \in \EE}$ and $\Gamma := (\gamma_{ij})_{(i,j) \in \EE}$ are the dual variables corresponding to the consensus constraints. Note that when forming the Lagrangian we have only dualized the consensus constraints. Denoting the feasibility set for each agent $i$ as
\begin{align*}
	\FF_i := \setdef{(y_i,z_i)}{y_i \in \YY, z_i \in \ZZ_i,  g_i(y_i,z_i,\xi) \le 0, \, \forall \xi \in \Xi},
\end{align*}
the distributed ADMM consists of the following updates
\begin{align*}
	(y_i^{k+1},z_i^{k+1}) & \gets \argmin_{(y_i,z_i) \in \FF_i}  \varphi_i(z_i) + h_i(y_i) 
	\\
	& \qquad + \sum_{j \in \NN_i} \bigl( (y_i^k)^\top (\lambda_{ij}^k + \gamma_{ij}^k) + \tfrac{\rho}{2} \norm{y_i - u_{ij}^k}^2 
	\\
	& \qquad \qquad + \tfrac{\rho}{2} \norm{y_i - u_{ji}^k}^2 \bigr), 
	\\
	u_{ij}^{k+1} &\gets \argmin_{u_{ij}} - u_{ij}^\top (\lambda_{ij}^k + \gamma_{ij}^k) + \tfrac{\rho}{2} \norm{u_{ij} - y_i^{k+1}}^2
	\\
	& \qquad + \tfrac{\rho}{2} \norm{u_{ij} - y_i^{k+1}}^2, 
	\\
	\lambda_{ij}^{k+1} &\gets \lambda_{ij}^k + \rho (y_i^{k+1} - u_{ij}^{k+1}),
	\\
	\gamma_{ij}^{k+1} & \gets \gamma_{ij}^k + \rho (y_j^{k+1} - u_{ij}^{k+1}).
\end{align*}
One could implement the above steps using only local communication. Nonetheless, with a particular initialization, some of the variables can be eliminated and the required communication and computation at each iteration can be reduced. Specifically, we do the following (see~\cite{Banjac:2019} and~\cite{GM-JAB-GBG:10} for details): (a) write the explicit expression for $u_{ij}^{k+1}$ using the first-order condition of optimality, (b) use this expression in the summation of $\lambda_{ij}^{k+1}$ and $\gamma_{ij}^{k+1}$ over the entire network to get $\lambda_{ij}^{k+1} + \gamma_{ij}^{k+1} = 0$, (c) initialize $\lambda_{ij}^0 = \gamma_{ij}^0 = 0$ and use the equality derived in (b) to define an update scheme for a new variable $p_i^k := 2 \sum_{j \in \NN_i} \lambda_{ij}^k$, and (d) rewrite the primal update $(z_i^{k+1},y_i^{k+1})$ using the variable $p_i^k$. This simplification reduces the ADMM into the update scheme
\begin{subequations}\label{eq:ADMM-dist-concise}
	\begin{align}
		p_i^{k+1} & \textstyle \gets p_i^k + \rho \sum_{j \in \NN_i} (y_i^{k} - y_j^{k}), \label{eq:p-update} 
		\\
		(y_i^{k+1},z_i^{k+1}) & \gets \argmin_{(y_i,z_i) \in \FF_i} \varphi_i(z_i) + h_i(y_i) + y_i^\top p_i^{k+1}  \notag
		\\
		& \textstyle \qquad \qquad + \rho \sum_{j \in \NN_i} \Bnorm{y_i - \tfrac{y_i^k + y_j^k}{2}}^2. \label{eq:primal-update}
	\end{align}
\end{subequations}
The above scheme can be implemented in a distributed manner where each agent $i$ at iteration $k$, broadcasts to its neighbors the variable $y_i^k$ and updates $(p_i^k,y_i^k,z_i^k)$. However, this scheme requires solving a semi-infinite program in~\eqref{eq:primal-update} at each iteration due to the definition of $\FF_i$. To overcome this limitation, we present the distributed cutting-surface ADMM method, formally given in Algorithm~\ref{alg:dist}, which combines the features of distributed ADMM and the cutting-surface algorithm explained in Section~\ref{sec:problem}.

\begin{quote}
\emph{[Informal description of Algorithm \ref{alg:dist}]:} In each iteration $k$ of this procedure, each agent $i$ executes an ADMM step (Lines~\ref{step:p-update} and~\ref{step:zy-update}). Different from~\eqref{eq:ADMM-dist-concise}, the feasibility set $\FF_i^k$ only imposes a finite number of constraints on the primal variables $(y,z)$, see  Line~\ref{step:F-update}. Each such constraint is defined by a point in the set $\Xi_i^k$ and is referred to as a cut. Line~\ref{step:cut-gen} finds an $\eps/2$-approximate maximizer $\xi_i^{k+1}$ of the constraint at the primal optimizer $(y_i^{k+1},z_i^{k+1})$ that was obtained in the ADMM step. If the constraint value at $(y_i^{k+1},z_i^{k+1},\xi_i^{k+1})$ is higher than $\eps/2$, then $\xi_i^{k+1}$ is added as a valid cut to the cut set $\Xi_i^k$. Line~\ref{step:F-update} updates the feasibility set. 
\end{quote}

In our distributed algorithm, agents form better outer approximations of the semi-infinite constraint set as the iterations progress. Note that as explained in the cutting-surface algorithm, we seek only an $\eps/2$-approximate maximizer for determining the cut as this maximization problem is nonconvex in general. The following result provides the convergence guarantees of Algorithm~\ref{alg:dist}. 

\begin{algorithm}[htb]
  \SetAlgoLined
  \DontPrintSemicolon
  \SetKwFor{Case}{case}{}{endcase}
  \SetKwInOut{giv}{Data} \SetKwInOut{ini}{Initialize}
  \SetKwInOut{state}{State} \SetKwInput{start}{Initiate}
  \SetKwInOut{msg}{Messages} \SetKwInput{Kw}{Executed by}
  \Kw{agents $i \in \until{n}$}
  \giv{agent $i$ has access to functions $\varphi_i$, $h_i$, $g_i$, sets $\YY$, $\Xi$, regularization parameter $\rho > 0$, and error tolerance $\eps >0$}
  \ini{agent $i$ sets iteration count $k=0$, variables $p_i^0 = 0$, $z_i^0 \in \ZZ_i$ $y_i^0 \in \YY$, cut set $\Xi_i^0 = \emptyset$, feasibility set $\FF_i^0 = \ZZ_i \times \YY$}
  \BlankLine \Repeat{termination condition is met}{ \tcc{Each agent $i$:} 
	  Exchange $y_i^k$ with neighbors $\NN_i$ \label{step1} \;  
	  Set $p_i^{k+1} \gets p_i^k + \rho \sum_{j \in \NN_i} (y_i^k - y_j^k)$ \label{step:p-update}  \; 
	  Set 
	  \vspace*{-0.5ex}
	  \begin{multline*} 
		  (y_i^{k+1},z_i^{k+1})  \gets \argmin_{(y_i,z_i) \in \FF_i^k}  \varphi_i(z_i) + h_i(y_i) 
		  \\
		   \quad + y_i^\top p_i^{k+1} + \rho \sum_{j \in \NN_i} \Bnorm{y_i - \tfrac{y_i^k + y_j^k}{2}}^2 
	  \end{multline*} 
	  \vspace*{-2.5ex}
		 \label{step:zy-update} \; 
	  Determine $\xi_i^{k+1} \in \Xi$ such that 
	  \vspace*{-0.5ex}
		  $$g_i(y_i^{k+1} ,z_i^{k+1},\xi_i^{k+1}) \ge \max_{\xi \in \Xi} \! g_i(y_i^{k+1},z_i^{k+1},\xi)  -  \frac{\eps}{2}$$
	  \vspace*{-2ex}
	  \label{step:cut-gen} \;
	  \eIf{$ g_i(y_i^{k+1},z_i^{k+1},\xi_i^{k+1}) >  \frac{\eps}{2}$}{Set $\Xi_i^{k+1}  \gets \Xi_i^k \cup \{\xi_i^{k+1}\}$ \label{step5}}{Set $\Xi_i^{k+1} \gets \Xi_i^k$ } \label{step:cut-gen-plus}
	  Set 
	  \vspace*{-1.5ex}
	  \begin{multline*}
		  \FF_i^{k+1} \gets \setdefbig{(z,y)}{z \in \ZZ_i, y \in \YY,  
	  \\
  g_i(y,z,\xi) \le 0, \, \forall \xi \in \Xi_i^{k+1}}
	\end{multline*}
	\vspace*{-2.5ex}
	\label{step:F-update} \; 
	Set $k \gets k+1$ \;}
  \caption{Distributed cutting-surface ADMM}\label{alg:dist}
\end{algorithm}

\begin{proposition}\longthmtitle{Convergence of distributed cutting-surface ADMM algorithm}\label{pr:conv}
	For Algorithm~\ref{alg:dist}, there exists a finite time $K \in \integerspositive$ such that $\Xi_i^{k_1} = \Xi_i^{k_2}$ for all $i \in \until{n}$ and all $k_1, k_2 \ge K$. Moreover, any limit point $(\bar{Y},\bar{Z}) :=(\bar{y}_i,\bar{z}_i)_{i \in \until{n}}$ of the sequence $\{(Y^k,Z^k):=(y_i^{k},z_i^{k})_{i \in \until{n}}\}_{k=0}^\infty$ satisfies $\bar{y}_i = \bar{y}_j$ for all $i,j \in \until{n}$, 
	\begin{align}\label{eq:termination-feas}
		\max_{i \in \until{n}} \, \max_{\xi \in \Xi} \, g_i(\bar{y}_i,\bar{z}_i,\xi) \le \eps,
	\end{align}
	and 
	\begin{align}\label{eq:final-opt}
	\sum_{i =1}^n \left( \varphi_i(\bar{z}_i) + h_i(\bar{y}_i) \right) \le J^\star,
	\end{align}
	where $J^\star$ is the optimum value of~\eqref{eq:gen-sip}.
\end{proposition}
\begin{proof}
	We use Lemma~\ref{le:finite-conv} to show first the existence of $K \in \integerspositive$ such that for all $k_1, k_2 \ge K$, we have 
	\begin{align}\label{eq:K-exist}
		\max_{i \in \until{n}} \, \max_{\xi \in \Xi} \, g_i(y_i^k,z_i^k,\xi) \le \eps,
	\end{align}
	$\Xi_i^{k_1} = \Xi_i^{k_2}$ for all $i \in \until{n}$. That is, after a finite number of iterations, no cuts are added in Line~\ref{step5} for any agent. Drawing the parallelism between the notation of Lemma~\ref{le:finite-conv} and the algorithm, pick an agent $i$ and note that the sequence $\setr{y_i^k,z_i^k}$ plays the role of $\setr{x^k}$, the set $\YY \times \ZZ_i$ that of $\XX$, and $g_i$ stands for the function $g$. By assumption, $g_i$ is locally Lipschitz and convex in $(z_i,y_i)$. Due to Line~\ref{step:zy-update}, we have 
	\begin{align*}
		g_i(y_i^{k+1},z_i^{k+1},\bar{\xi}) \le 0, \quad \forall \, \bar{\xi} \in \Xi^k_i.
	\end{align*}
	Further, the cut addition steps, Line~\ref{step:cut-gen} to Line~\ref{step:cut-gen-plus}, are same as part~\ref{cond:2} of the hypothesis of Lemma~\ref{le:finite-conv}. Therefore, as a result of Lemma~\ref{le:finite-conv}, we obtain~\eqref{eq:K-exist} and the fact that $\Xi^{k_1}_i = \Xi^{k_2}_i$ for all $i \in \until{n}$ and all $k_1, k_2 \ge K$. Note that due to~\eqref{eq:K-exist} and continuity, any accumulation point of the sequence $\{(Y^k,Z^k)\}$ satisfies~\eqref{eq:termination-feas}. Now consider any $\bar{k} > K$. 
	Consider the optimization problem 
	\begin{equation}\label{eq:gen-sip-with-cut}
		\begin{array}{cl}
			\underset{y,\setr{z_i}_{i=1}^n}{\min} & \displaystyle \sum_{i=1}^n \varphi_i(z_i) + h_i(y) 
			\\
			\st & y \in \YY, 
			\\
			& z_i \in \ZZ_i, \quad \forall i \in \until{n},
			\\ 
			& g_i(y,z_i,\xi) \le 0, \quad \forall \xi \in \Xi^{\bar{k}}_i, \quad i \in \until{n},
		\end{array}
	\end{equation}
	where the difference between this problem and~\eqref{eq:gen-sip} is that the semi-infinite constraint in~\eqref{eq:gen-sip} is replaced with a finite number of constraints in~\eqref{eq:gen-sip-with-cut}. After the $K$-th iteration, since the constraint set does not change, the algorithm reduces to the distributed ADMM algorithm that solves~\eqref{eq:gen-sip-with-cut} with consensus over the global variable $y$. Thus, from~\cite[Proposition 2]{GM-JAB-GBG:10}, for any limit point $(\bar{Y},\bar{Z}) =  (\bar{y}_i,\bar{z}_i)_{i \in \until{n}}$, we have consensus $\bar{y}_i = \bar{y}_j$ and $(\bar{y}_i,\bar{Z})$ is a solution of~\eqref{eq:gen-sip-with-cut}. Since the constraint set of~\eqref{eq:gen-sip} is a subset of the feasibility set of~\eqref{eq:gen-sip-with-cut} as $\Xi^{\bar{k}}_i \subset \Xi$ for all $i$, we conclude that~\eqref{eq:final-opt} holds.
\end{proof}

The above result shows that any accumulation point of the iterates maintained by the agents satisfies three properties. First, they achieve consensus over the global variable. Second, the semi-infinite constraint is satisfied with $\eps$ accuracy, and third, the objective value at the network-wide converged value is no more than the optimal value of the semi-infinite problem~\eqref{eq:gen-sip}. Thus, one can claim that~\eqref{eq:gen-sip} is solved up to an $\eps$ accuracy. Note that such an optimizer could be found in finite number of steps in the centralized cutting-surface method explained in Section~\ref{sec:sip-reform}. The convergence is asymptotic here due to the distributed nature of our algorithm.

\begin{remark}\longthmtitle{Generalizations of Algorithm~\ref{alg:dist}}
	In Algorithm~\ref{alg:dist}, the cut addition step can be skipped in some iterations by every agent. As long as one cut is added in every few iterations, the convergence still holds. On the same token, the guarantees are not affected if more than one cut is added at each iteration.  
	Furthermore, one need not start adding the cuts with $\eps$ tolerances. Initially, the tolerance can be high so that agents quickly reach consensus. Subsequently, agents can set the tolerance at a lower value to improve the accuracy of constraint set. Such modifications can possibly be used to improve the rate of convergence.  
	\oprocend
\end{remark}

\begin{remark}\longthmtitle{Comparison with literature}\label{re:comparison}
	The work~\cite{AC-JC:20-tac} gives a distributed saddle-point algorithm to solve the DRO problem~\eqref{eq:dro}. While the network structure and the availability of samples is same as our case, our work considerably generalizes the problem setup. Firstly, the method in~\cite{AC-JC:20-tac} relies on identifying a subset of $\XX \times \Xi$ where the function defining the semi-infinite constraint is concave in the uncertainty $\xi$. Our algorithm does not require such identification. Secondly, the objective function is assumed to be differentiable there while we can handle nonsmooth convex functions. Finally, we do not assume that agents have access to the number of other agents in the network or the total number samples.
	\oprocend
\end{remark}

\section{Distributed cutting-surface ADMM for DRO}\label{sec:dist-dro}

Here, we adapt the distributed algorithm presented in Algorithm~\ref{alg:dist} to the DRO problem~\eqref{eq:dro-sip}. To this end, we recall the comparison between~\eqref{eq:dro-sip} and~\eqref{eq:gen-sip}. In the former, $(x,s)$ are global variables on which all agents need to agree on and $\vag_i := (v_\ell)_{\setdef{\ell \in \until{L}}{\data{\ell} \in \Xihat_i}}$ is the local variable. In words, $\vag_i$ consists of decision variables $v_\ell$'s of problem~\eqref{eq:dro-sip} for which the data point $\data{\ell}$ is held by agent $i$. In problem~\eqref{eq:gen-sip}, $y$ is the global variable and $z_i$'s are the local ones. Functions $\varphi_i$ and $h_i$ in~\eqref{eq:gen-sip} can then be analogously written for~\eqref{eq:dro-sip} as 
\begin{align}\label{eq:varphi-dro}
	\vag_i \mapsto \ones_{\abss{\Xihat_i}}^\top \vag_i \quad \text{ and } \quad (x,s) \mapsto \abs{\Xihat_i} \theta s,
\end{align}
respectively. The constraint function $g_i$ for~\eqref{eq:dro-sip} reads as 
\begin{align}\label{eq:gi-dro}
	(x,s,\vag_i,\xi) \mapsto \max_{\setdef{\ell \in \until{L}}{\data{\ell} \in \Xihat_i}} f(x,\xi) - (\vag_i)_\ell - s \norm{\xi - \data{\ell}},
\end{align}
where $(\vag_i)_\ell$ refers to the component of the vector $\vag_i$ corresponding to the sample $\data{\ell}$. Finally, note that in problem~\eqref{eq:gen-sip} we assume that the global and local variables, for each agent $i$, are restricted to convex compact sets $\YY$ and $\ZZ_i$, respectively. However, while $x$ belongs to the convex compact set $\XX$ in~\eqref{eq:dro-sip}, we do not have such a constraint for variables $s$ and $v$. Thus, to use the results of the previous section, we find next a compact convex domain for~\eqref{eq:dro-sip} that does not disturb the optimizers. 

\begin{lemma}\longthmtitle{Compact domain DRO problem~\eqref{eq:dro-sip}}\label{le:compact}
	Any optimizer $(x^\star,s^\star,v^\star)$ of the DRO problem given in~\eqref{eq:dro-sip} belongs to the convex compact set $\XX \times \BB_s \times \BB_v$, where 
	\begin{align*}
		\BB_s := \left[ 0,\tfrac{\overline{f} - \underline{f}}{\theta} \right] \quad \text{ and }  \quad \BB_v = \II^L,
	\end{align*}
	with interval $\II := [\underline{f}, \underline{f} + L (\overline{f} - \underline{f})]$ and $\underline{f}$ and $\overline{f}$ are bounds such that $\underline{f} \le f(x,\xi) \le \overline{f}$ for all $(x,\xi) \in \XX \times \Xi$.
\end{lemma}
\begin{proof}
	Given $\Jref = L \Jdro$ where $\Jref$ and $\Jdro$ are optimal values of~\eqref{eq:dro-sip} and~\eqref{eq:dro}, respectively, we deduce that for any optimizer $(x^\star, s^\star, v^\star)$ of~\eqref{eq:dro-sip}, the following hods  
	\begin{align}\label{eq:opt-bound}
		\ones_{L}^\top v^\star + L \theta s^\star \le \Jref = L \Jdro \le L \overline{f}.
	\end{align}
	From feasibility of $(x^\star,x^\star,v^\star)$, we get for each $\ell \in \until{L}$,
	\begin{align*}
		f(x^\star,\xi) - v_\ell^\star - s^\star \norm{\xi - \data{\ell}} \le 0 \quad \forall \xi \in \Xi.
	\end{align*} 
	Setting $\xi = \data{\ell}$ in the above constraint gives us the inequality $v_\ell^\star \ge f(x^\star,\data{\ell}) \ge \underline{f}$. Further, using the feasibility condition $s^\star \ge 0$ in~\eqref{eq:opt-bound} yields $\sum_{\ell \in [L]} v_\ell^\star \le L \overline{f}$.	Using $v_\ell^\star \ge \underline{f}$ in this inequality gives the bound for each component of $v^\star$ as
	\begin{align*}
		v_\ell^\star & \le L \overline{f} - \sum_{k \in \until{L} \setminus \{\ell\}} v_k^\star
		\le L \overline{f} - (L-1) \underline{f}
	\end{align*}
	The above reasoning establishes that for any optimal $(x^\star,s^\star,v^\star)$, we have $v^\star \in \BB_v$. Finally, using the bound $v_\ell^\star \ge \underline{f}$ in~\eqref{eq:opt-bound} results in the upper bound $s^\star \le \frac{\overline{f} - \underline{f}}{\theta}$, thus establishing that $s^\star \in \BB_s$.
\end{proof}

Having described the parallelism between~\eqref{eq:dro-sip} and~\eqref{eq:gen-sip} and having identified the compact domains, we present Algorithm~\ref{alg:dist-dro} that solves the DRO problem.

\begin{center}
	\begin{fminipage}{0.47\textwidth}
		\begin{customthm}{2}\longthmtitle{Distributed cutting-surface ADMM for DRO}\label{alg:dist-dro}
			\rm{
				The procedure involves executing Algorithm~\ref{alg:dist} with the following elements: \\
				\textbf{Parameters:} regularization $\rho > 0$ and  tolerance $\eps >0$
				
				\textbf{Variables:} $(\xag_i,\sag_i) \in \real^{d+1}$ as $y_i$; $\vag_i \in \real^{\abss{\Xihat_i}}$ as $z_i$; and the auxiliary variable $p_i \in \real^{d+1}$ remains same 
				
			\textbf{Sets:} $\XX \times \BB_s$ as $\YY$; $\II^{\abss{\Xihat_i}}$ as $\ZZ_i$ (see Lemma~\ref{le:compact} for definitions); and the uncertainty set $\Xi$ remains the same
				
				\textbf{Functions:} maps in~\eqref{eq:varphi-dro} as $\varphi_i$ and $h_i$ and map~\eqref{eq:gi-dro} as $g_i$

			}
		\end{customthm}
	\end{fminipage}
\end{center}

In the above algorithm, agents need to agree beforehand upon the error tolerance $\eps$ and the regularization parameter $\rho$. The former of these can be selected independently by each agent and the algorithm still converges. The guarantee on the obtained optimizer depends on the largest of these tolerance values. The following result summarizes the convergence guarantees of the algorithm. It uses the convergence result of the previous section and the fact that if one satisfies the semi-infinite constraints in~\eqref{eq:dro-sip} with an $\eps$ error and optimizes the objective function under these constraints, then one obtains an $\eps$-optimal solution of the DRO problem~\eqref{eq:dro}.
\begin{proposition}\longthmtitle{Convergence of distributed cutting-surface ADMM for DRO}\label{pr:conv-dist}
	For Algorithm~\ref{alg:dist-dro}, there exists a finite time $K \in \integerspositive$ such that $\Xi_i^{k_1} = \Xi_i^{k_2}$ for all $i \in \until{n}$ and all $k_1, k_2 \ge K$. Furthermore, for any limit point $(X,S,V)=(\xag_i,\sag_i,\vag_i)_{i \in \until{n}}$ of the sequence $\{(X^k,S^k,V^k):=((\xag_i)^{k}, (\sag_i)^k, (\vag_i)^{k})_{i \in \until{n}}\}_{k=0}^\infty$ the following is true:
	\begin{enumerate}
		\item $(\xag_i,\sag_i) = (\xag_j,\sag_j)$ for all $i,j \in \until{n}$,
		\item denoting $v^\star$ as the collection $(\vag_i)_{i \in [N]}$ and $s^\star = \sag_i$, the inequality below holds 
		\begin{align}\label{eq:ineq-dro}
			\ones_{L}^\top v^\star + L \theta s^\star \le L \Jdro,
		\end{align}
		where $\Jdro$ is the optimal value of~\eqref{eq:dro}, and
		\item denoting $x^\star = \xag_i$, we have 
		\begin{align}\label{eq:eps-opt-dro}
			\sup_{\Qb \in \aset} \Eb_{\Qb} [f(x^\star,\xi)] \le \Jdro + \eps.
		\end{align}
	\end{enumerate} 	
\end{proposition}
\begin{proof}
	From Proposition~\ref{pr:conv}, we deduce that (a) after a finite number of steps, no cuts are added, (b) at any limit point of the algorithm, we achieve consensus over the variables $(\xag_i,\sag_i)$, (c) for every agent $i \in \until{n}$, 
	\begin{align*}
		\max_{\xi \in \Xi} g_i(\xag_i,\sag_i,\vag_i,\xi) \le \eps,
	\end{align*}
	and finally (d) denoting $v^\star = (\vag_i)_{i \in \until{n}}$ and $s^\star = \sag_i$, the inequality~\eqref{eq:ineq-dro} holds as $\ones_L^\top v^\star + L \theta s^\star \le \Jref = L \Jdro$, where recall that $\Jref$ is the optimum value of~\eqref{eq:dro-sip}. Lastly, as a consequence of (c) and (d), and employing the arguments of the proof of~\cite[Theorem 3.2]{FL-SM:19}, we arrive at~\eqref{eq:eps-opt-dro}.	
\end{proof}
The above result implies that asymptotically, each agent arrives at a point $x$ that is $\eps$-optimal for the DRO problem~\eqref{eq:dro}. That is, the cost incurred at $x$ is no more than $\eps$ higher than the optimal value of~\eqref{eq:dro}. A natural question to ask is if one can reduce $\eps$ along the execution of the algorithm to arrive at the exact optimizer. We believe this is possible and we wish to explore it  in future. The ensuing remark demonstrates how our distributed algorithm can be used for solving distributionally robust risk minimization.

\begin{remark}\longthmtitle{Distributed distributionally robust risk minimization}
	{\rm
		Consider the following risk minimization problem
		\begin{align}\label{eq:dr-cvar}
			\inf_{x \in \XX} \sup_{\Qb \in \aset} \CVaR_{\beta}^{\Qb} [ f(x,\xi)],
		\end{align}
		where $\beta \in (0,1)$ is the risk-averseness parameter and $\CVaR$ is the conditional value-at-risk~\cite{AS-DD-AR:14}. Roughly speaking, the quantity $\CVaR_\beta^\Qb[f(x,\xi)]$ stands for the expectation of the tail of the random variable $\xi \mapsto f(x,\xi)$ that has $\beta$ mass under the distribution $\Qb$. Mathematically, we have 
		\begin{align}\label{eq:cvar}
			\CVaR_{\beta}^\Qb [f(x,\xi)] = \inf_{t \in \real} \Bigl\{ t + \frac{1}{\beta} \Eb_\Qb [f(x,\xi) - t]_+ \Bigr\},
		\end{align}
		where for $u \in \real$, we have $[u]_+ := \max \{u,0\}$.  Problem~\eqref{eq:dr-cvar} minimizes the worst-case risk of the uncertain function $\xi \mapsto f(x,\xi)$ over the ambiguity set and so, is a generalization of the DRO problem~\eqref{eq:dro}. Nonetheless, Algorithm~\ref{alg:dist} can be used to solve~\eqref{eq:dr-cvar} in a distributed manner. Specifically, assume the same network and information structure as in Section~\ref{sec:problem}. To arrive at a distributed algorithm, we reformulate~\eqref{eq:dr-cvar} and bring it into the form of~\eqref{eq:gen-sip}. For a fixed $x \in \XX$, we have 
		\begin{align*}
			\sup_{\Qb \in \aset} \! \CVaR_{\beta}^{\Qb} [ f(x,\xi)] & \!  = \!  \sup_{\Qb \in \aset} \inf_{t \in \real} \Bigl\{ t + \frac{1}{\beta} \Eb_\Qb [f(x,\xi) - t]_+ \Bigr\}
			\\
			& \!  = \!  \inf_{t \in \real} \sup_{\Qb \in \aset} \Bigl\{ t + \frac{1}{\beta} \Eb_\Qb [f(x,\xi) - t]_+ \Bigr\},
		\end{align*}
		where the last equality follows from~\cite[Lemma 3.2]{AC-ARH:21}. Using the above equality, problem~\eqref{eq:dr-cvar} is equivalent to 
		\begin{align}\label{eq:cvar-t}
			\inf_{x \in \XX, t \in \real} \sup_{\Qb \in \aset} \,  \Bigl\{ t + \frac{1}{\beta} \Eb_\Qb [ f(x,\xi) -t]_+ \Bigr\}.
		\end{align}
	For any $x \in \XX$ and $\Qb \in \aset$, the optimizer of  problem~\eqref{eq:cvar} belongs to the compact set $\TT := [\underline{f},\overline{f}] \subset \real$, where this interval is such that $f(x,\xi) \in \TT$ for all $(x,\xi) \in \XX \times \Xi$, see~\cite[Chapter 6]{AS-DD-AR:14} for details. Hence, one can restrict the domain of the optimization problem~\eqref{eq:cvar-t} to $\XX \times \TT$ without disturbing the optimizers. Subsequently, using ~\cite[Theorem 3.1]{FL-SM:19}, we reformulate~\eqref{eq:cvar-t} as 
		\begin{align*}
			\underset{(x,t,s,v)}{\min} & \quad Lt + \frac{1}{\beta} \Bigl( \ones_L^\top v + L \theta s \Bigr) 
			\\
			\st & \quad v \in \real^{L}, \quad s \in \realnonnegative, \quad x \in \XX, \quad t \in \TT, 
			\\
			& \quad [f(x,\xi) - t]_+ + \beta t - \beta v_\ell - \beta s \norm{\xi - \data{\ell}} \le 0,
			\\
			& \qquad \qquad  \qquad \qquad \qquad \qquad \forall \xi \in \Xi, \;\; \ell \in \until{L}.
		\end{align*} 
	The above problem has the structure of~\eqref{eq:gen-sip} where the global variable is $(x,t,s)$ and the local ones are the components of $v$ corresponding to the samples that each agent holds. Further, the local and global objective function for agent $i$ would be $\vag_i \mapsto \beta^{-1} \ones_{|\Xihat_i|}^\top \vag_i$ and $(x,t,s) \mapsto |\Xihat_i|(\theta \beta^{-1} s + t)$, respectively. The constraints can be figured out in a similar manner. Hence, Algorithm~\ref{alg:dist} can be adapted to solve the above problem in a distributed way. As an intermediate step, one needs to identify a compact set in which optimizers of the above problem lie. This can be done much like the way done in Lemma~\ref{le:compact} using compactness of $\XX\times\TT$.
	}
		\oprocend
\end{remark}

\section{Simulations}\label{sec:sim}
Here, we demonstrate the application of Algorithm~\ref{alg:dist-dro} to a distributionally robust linear regression problem. Consider $n=6$ agents communicating over a graph defined by an undirected ring with additional edges $\{(1,4),(2,3)\}$. Each data point $\xihat^k = (\uhat^k,\vhat^k)$ consists of an input vector $\uhat^k \in \real^4$ and a scalar output $\vhat^k$. For constructing the dataset, we assume that $\uhat^k$ is sampled from multivariate normal distribution with mean zero and covariance as the identity matrix. The output takes value $\vhat^k = (1,2,3,1)^\top\uhat + \nuhat^k$, where $\nuhat^k$ is uniformly sampled from the interval $[-1,1]$. We assume that each agent has access to $10$ samples of this form. Note that the input-output pairs have an affine relationship defined by the vector $w = (1,2,3,1,0)^\top$. That is, $\vhat^k = w^\top(\uhat^k,1)^\top + \vhat^k$. The aim for the agents is to identify a vector $x^\star$ that explains the affine relationship with desirable out-of-sample guarantees. To this end, agents wish to solve the following DRO problem
\begin{align}\label{eq:dro-sim}
	\inf_{x \in \XX} \sup_{\Qb \in \aset} \Eb_{\Qb} [ f(x,\xi)]
\end{align} 
where $x \in \real^5$, $\XX := [0,5]^5$, and $f(x,\xi) := (v - x^\top (u;1))^2$ is the quadratic loss that determines the prediction accuracy of $x$ for input-output pair $\xi = (u,v)$. The radius of the ambiguity set is $\theta = 0.01$. Recall that $L = 60$ as we have six agents and each agent has $10$ samples. For our distributed algorithm, we set the error tolerance and regularizer as $\eps = 0.01$ and $\rho =  0.05$, respectively. Note that the maximization problems solved in each iteration by each agent (line~\ref{step:cut-gen} in Algorithm~\ref{alg:dist}) is nonconvex in our case as $f$ is convex in $\xi$. Thus, we make use of the branch-and-bound algorithm proposed in~\cite[Algorithm 2]{FL-SM:19} to find an $\eps/2$ approximate solution to the problem.  Figure~\ref{fig:sim1} depicts the execution of our algorithm. As shown, the agents achieve consensus over the global variable and find the solution of the DRO problem~\eqref{eq:dro-sim} as $x^\star = (0.89, 1.84, 2.89, 0.83, 0.06)$.

\begin{figure*}
	\centering
	\begin{subfigure}[b]{0.3\linewidth}
		\centering
		\includegraphics[width=\linewidth]{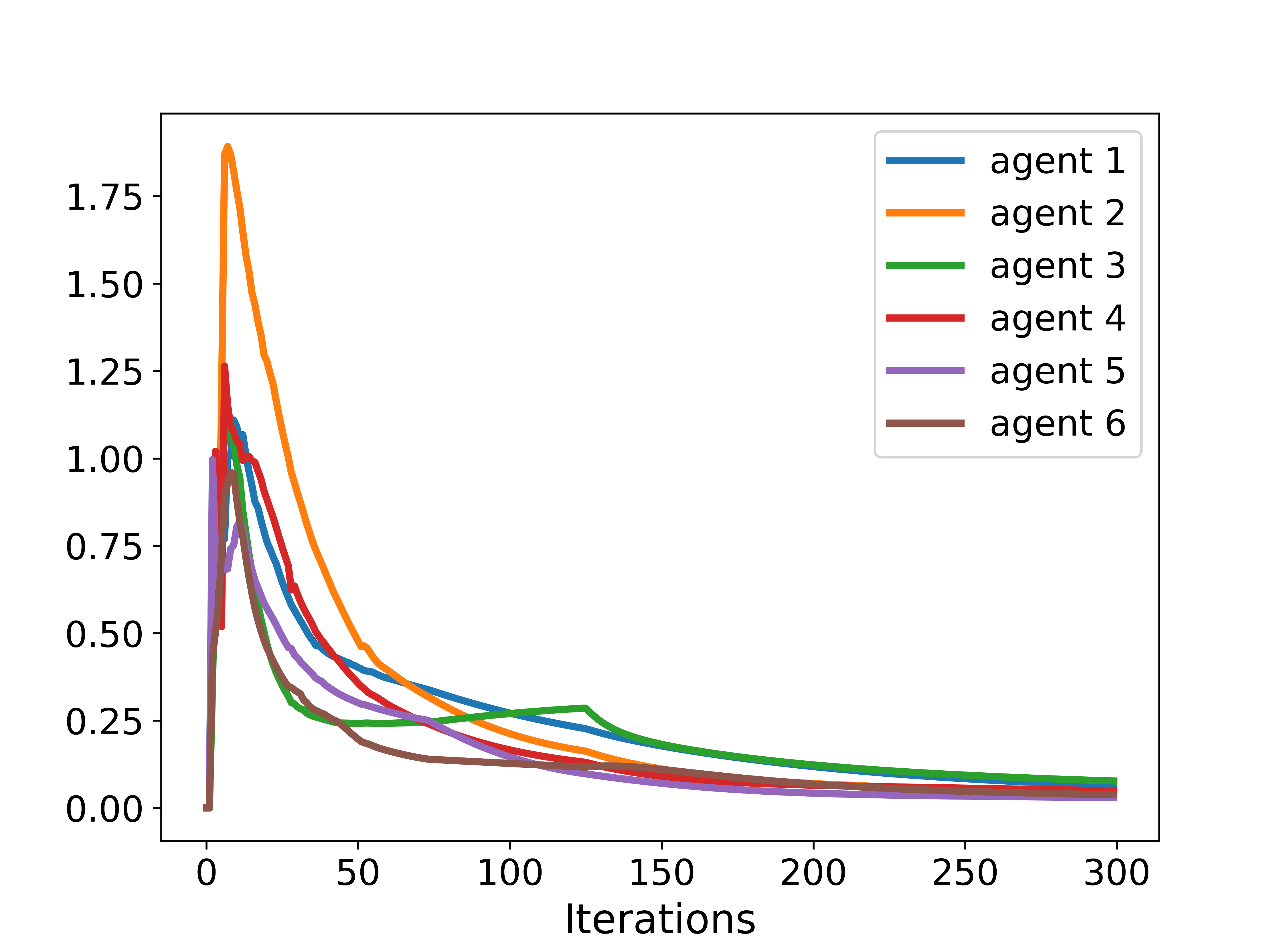}
		\caption[Network2]{{ \footnotesize $\norm{x_{\text{avg}} - x_i}$ }}%
		{{\small }}   
		\label{fig:trajs_a}
	\end{subfigure}
	\begin{subfigure}[b]{0.3\linewidth}  
		\centering 
		\includegraphics[width=\linewidth]{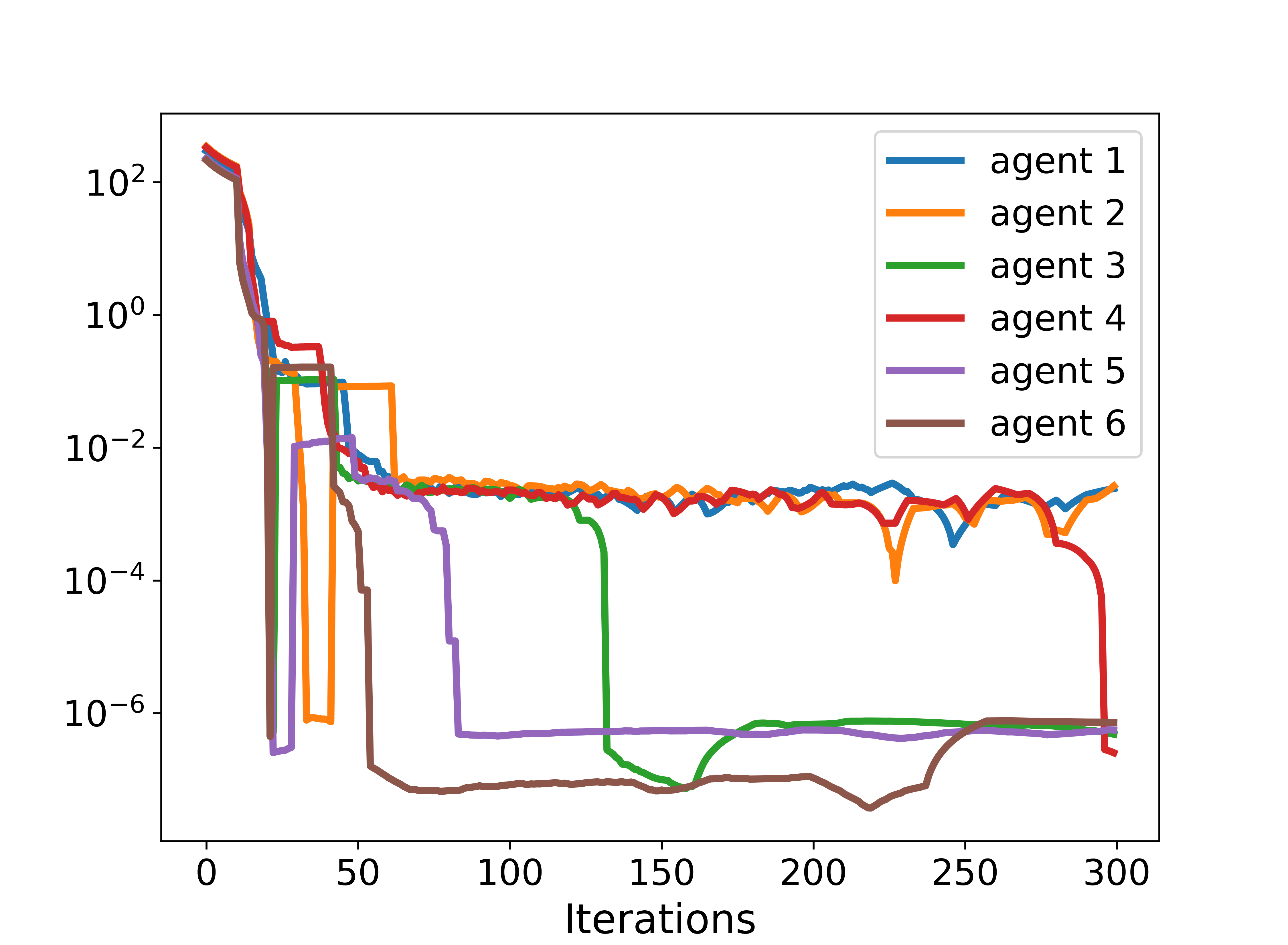}
		\caption[]{{\footnotesize $\max_{\xi \in \Xi} g_i(\xag_i,\sag_i,\vag_i,\xi)$} }%
		{{\small }}
		\label{fig:trajs_b}
	\end{subfigure}
	\begin{subfigure}[b]{0.3\linewidth}   
		\centering 
		\includegraphics[width=\linewidth]{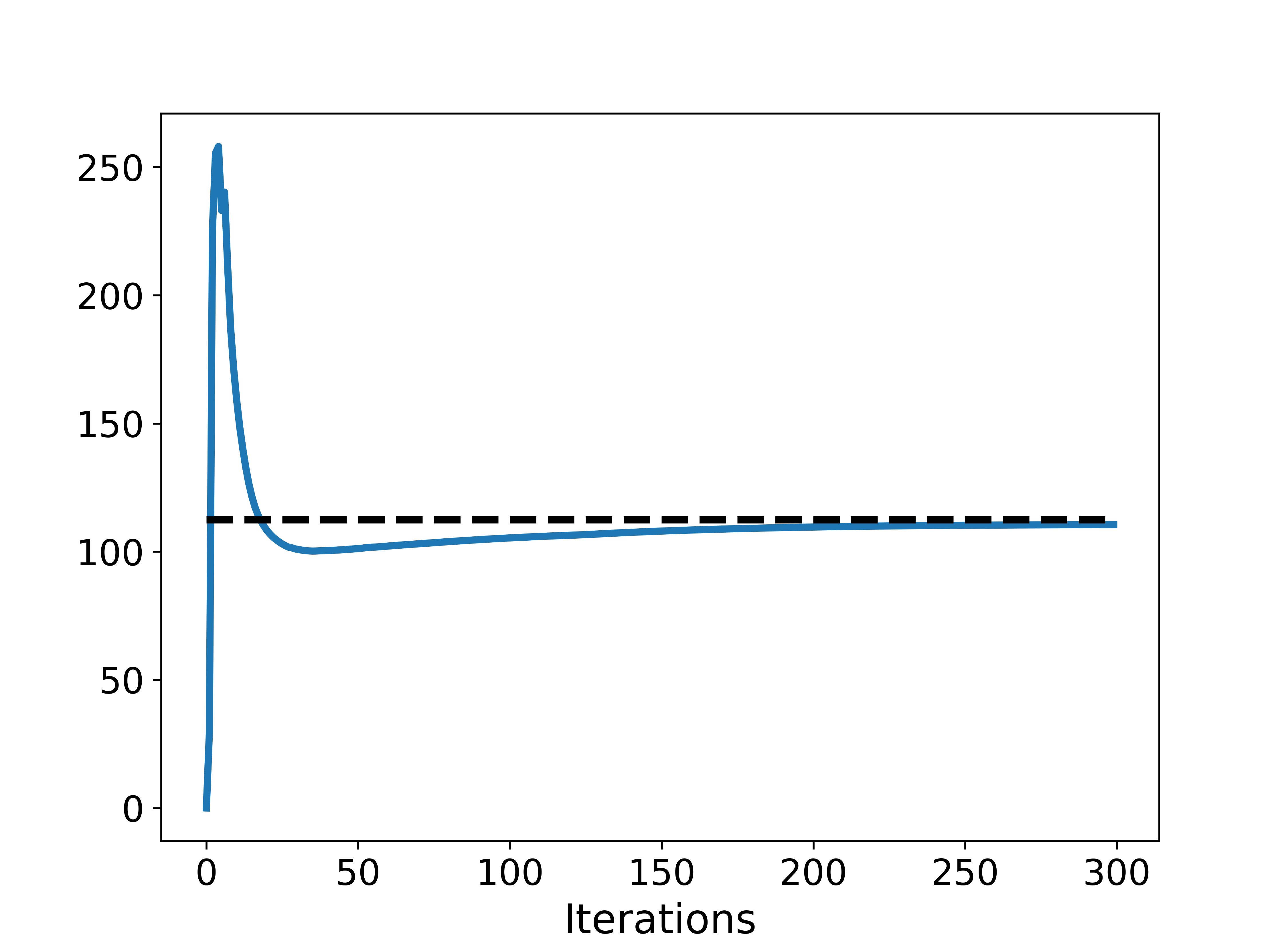}
		\caption[]{{\footnotesize $\sum_{i=1}^6 \ones_{|\Xihat_i|} \vag_i + |\Xihat_i| \theta \sag_i$}}%
		{{\small }}    
		\label{fig:trajs_c}
	\end{subfigure}
	\caption{Illustration of the execution of the distributed cutting-surface ADMM algorithm for solving the DRO problem~\eqref{eq:dro-sim} given in Section~\ref{sec:sim}. Plot (a) depicts consensus over the variable $x$. Here we see that the iterates $x_i$ maintained by each agent $i$ converge to the average value $x_{\text{avg}} = (1/6) \sum_{i=1}^6 x_i$. Plot (b) shows the constraint violation for the primal iterates $(\xag_i,\sag_i,\vag_i)$ for each agent. That is, it evaluates $\max_{\xi \in \Xi} g_i(\xag_i,\sag_i,\vag_i,\xi)$ at each iteration where $g_i$ is given in~\eqref{eq:gi-dro}. It shows that in less than hundred iterations, all agents suffer from no more than $\eps = 0.01$ constraint violation. Finally, plot (c) portrays that the value of the objective function of the reformulated DRO problem~\eqref{eq:dro-sip} written as the summation over agents $\sum_{i=1}^6 \ones_{|\Xihat_i|} \vag_i + |\Xihat_i| \theta \sag_i$ converges to a value less than $\Jref = L \Jdro$ (shown by dark dashed line), as asserted in~\eqref{eq:ineq-dro}.  }
	\vspace*{-2ex}
	 \label{fig:sim1}
\end{figure*}

\section{Conclusions}
We have designed and analyzed a distributed algorithm that solves the data-driven DRO problem to a pre-specified arbitrary accuracy. Our algorithm combines the features of cutting-surface algorithm and distributed ADMM. There are several open future research directions. First, we aim to analyze the convergence of our algorithm to the optimum when the accuracy is improved along the iterations. Next, we plan to design finite-memory algorithms that can solve the DRO problem to a certain approximation even when the sample size is large. Finally, we intend to explore primal-dual distributed algorithms to handle semi-infinite constraints.

\section*{Appendix}
\renewcommand{\theequation}{A.\arabic{equation}}
\renewcommand{\thetheorem}{A.\arabic{theorem}}
The next result aids us in proving Proposition~\ref{pr:conv}. The proof is inspired by that of~\cite[Theorem 7.2]{RH-KOK-93} where cuts were added whenever the semi-infinite constraint is violated (without considering the $\eps$ error tolerance). Our result shows that as long as the semi-infinite constraints are approximated iteratively with increasing number of cuts and the primal variable, at any iteration, satisfies the constraint for all cuts added in previous iterations, then after a finite time, primal variable always satisfies the semi-infinite constraint approximately and no more cuts are added.

\begin{lemma}\longthmtitle{Finite termination of cut addition process}\label{le:finite-conv}
	Consider the problem $\max_{\xi \in \Xi} g(x,\xi)$, where $\map{g}{\real^{d} \times \real^{m}}{\real}$ is locally Lipschitz. Assume that $g$ is convex in the first argument $x$ for any fixed $\xi \in \Xi$ and $\Xi \subset \real^m$ is a compact set. Let $\eps > 0$, $\XX \subset \real^d$ be a compact set, and consider sequences $\setr{x^k}_{k=0}^\infty \subset \XX$ and $\setr{\Xi^k}_{k=0}^\infty$ constructed iteratively as: $x^0 \in \XX$, $\Xi^0 = \emptyset$, and for all $k \in \integerspositive$,
	\begin{enumerate}
		\item \label{cond:1} (primal feasibility): find $x^{k+1} \in \XX$ satisfying
			\begin{align}\label{eq:g-less-zero}
				g(x^{k+1},\bar{\xi}) \le 0, \qquad \forall \, \bar{\xi} \in \Xi^k;
			\end{align}
		\item \label{cond:2} 
			(cut addition): set $\Xi^{k+1} = \Xi^k \cup \setr{\xi^{k+1}}$, where the point $\xi^{k+1} \in \Xi$ satisfies 
			\begin{align}\label{eq:approx-opt}
				g(x^{k+1},\xi^{k+1}) \ge \max_{\xi \in \Xi} g(x^{k+1},\xi) - \frac{\eps}{2},
			\end{align}
			and $g(x^{k+1},\xi^{k+1}) > \eps/2$. If no such point $\xi^{k+1}$ exists, then $\Xi^{k+1} = \Xi^k$.
	\end{enumerate}
	Then, there exists a $K \in \integerspositive$ such that $\Xi^{k_1} = \Xi^{k_2}$ for all $k_1, k_2 \ge K$ and 
	\begin{align}\label{eq:K-cond}
		\max_{\xi \in \Xi} g(x^k,\xi) \le \eps, \qquad \forall k \ge K.
	\end{align}
\end{lemma} 
\begin{proof}
	We first show the existence of $K \in \integerspositive$ such that $\Xi^{k_1} = \Xi^{k_2}$ for all $k_1, k_2 \ge K$. To this end, assume the contrary. This implies the existence of a sequence $\setr{(x^{k_n}, \xi^{k_n})}_{n=1}^\infty$ such that $\xi^{k_n} \in \Xi^{k_n} \setminus \Xi^{k_n - 1}$ for all $n$. That is, at iteration $k_n - 1$, the cut $\xi^{k_n}$ is added to the set $\Xi^{k_n - 1}$. By conditions~\ref{cond:1} and~\ref{cond:2} of the hypotheses, we get $g(x^{k_n}, \xi^{k_n}) > \eps/2$ and $g(x^{k_n},\xi^{k_q}) \le 0$ for all $q \in \until{n-1}$. That is, 
		\begin{align}\label{eq:g-diff-bound}
			g(x^{k_n},\xi^{k_n}) - g(x^{k_n},\xi^{k_q}) > \frac{\eps}{2}, \quad \forall q \in [n-1].
		\end{align}
	By hypothesis, $g$ is Lipschitz with some constant $L >0$ over the compact set $\XX \times \Xi$. This property and the condition~\eqref{eq:g-diff-bound} imply that $\norm{\xi^{k_n} - \xi^{k_q}} \ge \frac{\eps}{4L}$ for all $q \in \until{n-1}$. Indeed, otherwise, if there exists some $\xi^{k_q}$ such that $\norm{\xi^{k_n} - \xi^{k_q}} \le \frac{\eps}{4 L}$, then we contradict~\eqref{eq:g-diff-bound} as 
		\begin{align*}
			\abs{ g(x^{k_n}, \xi^{k_n}) - g(x^{k_n},\xi^{k_q}) } \le L \norm{\xi^{k_{n}} - \xi^{k_q}} \le \frac{\eps}{4}.
		\end{align*}
		Summarizing the above reasoning, we have that the infinite set of points $\{\xi^{k_n}\}_{n=1}^\infty$ are at least $\frac{\eps}{4 L}$ distant from each other. This is a contradiction as $\Xi$ is a compact set, see~\cite[Theorem 3.1 and Corollary 3.4]{VT:11}. Thus, we reach the conclusion that after the $K$-th iteration no cuts are added. Finally, this implies~\eqref{eq:K-cond} because otherwise if $x^k$ for some $k > K$ violates $\max_{\xi \in \Xi} g(x^k,\xi) \le \eps$, then one can find $\xi^{k}$ such that $g(x^k,\xi^{k}) \ge \max_{\xi \in \Xi} g(x^k,\xi) - \eps/2$ and $g(x^k,\xi^k) > \eps/2$. This would mean $\xi^k$ is added to the cut set which again leads to a contradiction. This completes the proof.
\end{proof}

Notice that in the above result, at each iteration, the primal variable $x^{k+1}$ merely needs to be feasible for the added cuts in $\Xi^k$; there is no requirement of it being optimal with respect to some other function. This is different from the cutting-surface algorithm~\cite[Algorithm 1]{FL-SM:19}, where each $x^k$ of the primal sequence is an optimizer. Additionally, we highlight that at each iteration, we can add more than one cut in the above result and the statement still holds as long as there is one cut that satisfies the conditions in~\ref{cond:2}. 

\bibliographystyle{ieeetr}

\end{document}